\documentclass[12pt, letterpaper]{article} 

\usepackage{amsmath,amsfonts,amssymb,mathtools}

\usepackage{remexpps,pproof} 

\usepackage[all]{xy}
\SelectTips{cm}{}

\def\today{\number\day\space\ifcase\month\or
   January\or February\or March\or April\or May\or June\or
   July\or August\or September\or October\or November\or December\fi 
   \space\number\year}

\usepackage{accents}
\newcommand{\pbar}[1]{\begingroup\let\hidewidth\relax%
\accentset{\hrulefill}{#1}\endgroup}
\renewcommand{\bar}[1]{\pbar{#1}}
   
\newtheorem{theorem}{Theorem}[section]

\newtheorem{lemma}[theorem]{Lemma}

\newtheorem{corollary}[theorem]{Corollary}
\newremark{definition}{Definition}[section]
\newremark{example}{Example}[section]
\newremark{remark}{Remark} [section]

\usepackage{mathrsfs}

\let\mathscr=\relax

\usepackage{eucal}
\let\escr=\mathcal
\let\mathcal=\relax

\usepackage[all]{xy}
\SelectTips{cm}{}

\newcommand{\abs}[1]{\left\vert #1 \right\vert}

\newcommand{\br}[2]{\left[#1,#2\right]}
\newcommand{\bre}{\br{\ }{\,}}

\newcommand{\C}{\mathbb{C}}

\newcommand{\dg}{\dot{\g}}
\newcommand{\DGS}{D{\kern-.375em}G{\kern-.2em}S}

\newcommand{\ds}{\oplus}

\newcommand{\End}{\mathrm{End}}

\newcommand{\g}{\gamma}

\newcommand{\id}{\mathrm{Id}}

\newcommand{\inp}[2]{\left\langle #1, #2 \right\rangle}
\newcommand{\inpe}{\inp{\ }{\,}}
\newcommand{\Isa}{\Iso^\mathrm{aut}}

\newcommand{\Iso}{\mathrm{Iso}}

\newcommand{\Isp}{\Iso^\mathrm{spl}}

\newcommand{\lah}{\lal{h}}
\newcommand{\lal}[1]{\mathfrak{#1}}
\newcommand{\lav}{\lal{v}}
\newcommand{\laz}{\lal{z}}
\newcommand{\n}{\lal{n}}
\newcommand{\nin}{\noindent}
\newcommand{\norm}[1]{\left\Vert #1 \right\Vert}
\newcommand{\norme}{\norm{\cdot}}

\newcommand{\R}{\mathbb{R}}

\newcommand{\rc}{\mathrm{Rc}}

\newcommand{\ric}[2]{\mathrm{Ric}(#1,#2)}

\newcommand{\scp}[2]{\left(#1,#2\right)}
\newcommand{\scpe}{\scp{\ }{\,}}
\newcommand{\scpp}[2]{\scp{#1}{#2}_\vp}
\newcommand{\scppe}{\scpe_\vp}

\renewcommand{\span}[1]{\text{span}\left\{ #1 \right\}}
\newcommand{\surj}{\rightarrow\kern-.82em\rightarrow}

\newcommand{\ve}{\varepsilon}

\newcommand{\vp}{\varphi}

\makeatletter
\newcommand{\ad}[1]{\mathop{\operator@font ad}\nolimits_{#1}}
\newcommand{\add}[1]{\mathop{\operator@font ad}^\dagger\nolimits_{#1}}
\newcommand{\adj}[1]{\mathop{\operator@font ad}^*\nolimits_{#1}}
\newcommand{\del}[1]{\mathop{\nabla}\nolimits_{\mspace{-2.5 mu} #1}}
\newcommand{\tr}[1]{\mathop{\operator@font tr}\nolimits\left(#1\right)}
\makeatother

\begin{document}
\thispagestyle{empty}

\nin {\Huge
\bf {\sf
Pseudo-Riemannian Lie groups \\[0.5ex] of modified $H$-type
}}

\vspace{1.5 ex}

\nin Justin M. Ryan%

\




\nin \begin{quote} {\small 
{\bf Abstract.}
We define a class of Riemannian and pseudo-Riemannian 2-step nilpotent
Lie groups with nondegenerate centers (Definition \ref{gh}) that generalize the
$H$-type groups of Kaplan \cite{K1,K2,K3}. Examples are given and geometric
properties are investigated. 
}
\end{quote}

\hfill


\section{Introduction}


Let $N$ be a 2-step nilpotent Lie group with Lie algebra $\n$. Let $\bre$ denote the
Lie bracket on $\n$, and write $\n = \laz \ds \lav$, where $\laz$ is the center of 
$\n$. Let $\inpe$ denote
both a pseudo-Riemannian inner product on $\n$ making $\laz$ nondegenerate, and
the induced left-invariant metric tensor on $N$. 

The adjoint representation of $\n$ on itself, $\ad{} :\n \times \n \to \n$, is
given by $\ad{x} = \br{x}{\ }$. For every $z \in \laz$ define a skew-symmetric
linear transformation $j(z)$ on $\lav$ by
$$
j(z) = \adj{\bullet}z,
$$
where $\adj{}$ is the adjoint of $\ad{}$ with respect to $\inpe$. That is,
$$
\inp{\ad{x}y}{z} = \inp{y}{j(z)x}
$$
for all $x, y \in \lav$ and all $z \in \laz$.

These $j$-maps simultaneously encode the algebraic and geometric information
about the (pseudo-) Riemannian 2-step nilpotent Lie group $(N,\inpe)$. There is
a special family of groups, originally studied by Kaplan \cite{K2}, for which 
these maps act like scaled anti-involutions.

\begin{definition}
A\label{htype} 2-step nilpotent Lie group $N$ with left-invariant Riemannian
metric $\inpe$ is said to be of \emph{$H$-type} if and only if for each $z
\in \laz$,
$$
\inp{j(z)x}{j(z)y} = \norm{z}^2 \inp{x}{y}
$$
for all $x, y \in \lav$. Alternatively, for each $z \in \laz$,
$j(z)^2 = - \norm{z}^2 \id|_\lav$.
\end{definition}

\begin{definition}
A\label{phtype} pseudo-Riemannian 2-step nilpotent Lie group $N$ with
nondegenerate center is said to be of
\emph{pseudo-$H$-type} if and only if for each $z \in \laz$,
$$
\inp{j(z)x}{j(z)y} = \norm{z}^2 \inp{x}{y}
$$
for all $x, y \in \lav$. Alternatively, for each $z \in \laz$,
$j(z)^2 = - \norm{z}^2 \id|_\lav$.
\end{definition}

\section{Definition and Examples}

Definitions\label{de} \ref{htype} and \ref{phtype} can be generalized by replacing
the quadratic form determined by $\inpe_\laz$ with another quadratic form on $\laz$.
J. Lauret \cite{L99} had this idea in the Riemannian case, and gave the following 
definition.

\begin{definition}
A\label{mhriem} Riemannian 2-step nilpotent Lie group is said to be a 
\emph{modified $H$-type group} if for any nonzero $a \in \laz$, $j(a)^2 = \lambda(a)
\id|_\lav$ for some $\lambda(a) < 0$.
\end{definition}

\emph{A priori}, the $\lambda$ in this definition is just a function on $\laz$.
However, the properties of the $j$-maps actually imply that $\lambda$ is a 
nondegenerate, negative-definite quadratic form on $\laz$.

We generalize this definition by not only considering pseudo-Riemannian metrics on
$N$, but also allowing quadratic forms on $\laz$ of the
most general type. That is, we allow that the associated symmetric bilinear
form be indefinite and/or degenerate. 

\begin{definition}
Let\label{gh} $(N,\inpe)$ be a 2-step nilpotent Lie group with left-invariant metric
$\inpe = \inpe_\laz + \inpe_\lav$ making $\laz$ nondegenerate, and let $\vp$
be a quadratic form on $\laz$.
We shall say that $(N,\inpe,\vp)$ is of \emph{modified $H$-type}
if and only if for each $z \in \laz$, the map $j(z) \in \End(\lav)$ satisfies
\begin{equation}
\inp{j(z)x}{j(z)y}_\lav\label{ghj} = \vp(z)\inp{x}{y}_\lav
\end{equation}
for all $x, y \in \lav$. Alternatively, since the $j$-maps are skew-adjoint this
property is characterized by 
\begin{equation}
j(z)^2\label{ghjs} = -\vp(z)\,\id|_\lav
\end{equation}
for all $z \in \laz$.
\end{definition}

\begin{remark}
Even when the metric is Riemannian, Definition \ref{gh} is more general (although
not necessarily more interesting) than 
Definition \ref{mhriem}. This is illustrated below by Example \ref{teh0}.
\end{remark}

\begin{example}
If\label{gen} $\inpe$ is a left-invariant Riemannian metric on $N$ and $\vp =
\norme_\laz^2$,
then $(N,\inpe)$ is an $H$-type group in the sense of Kaplan \cite{K2}. If $\inpe$
is a left-invariant pseudo-Riemannian metric on $N$ and $\vp = \norme_\laz^2$, then
$(N,\inpe)$ is a pseudo-$H$-type group in the sense of Ciatti \cite{Ci}.
\end{example}

\begin{example}
Consider\label{h3sig} the Heisenberg algebra $\lah_3$ spanned by the vectors
$\{e_1,e_2,e_3\}$ with nontrivial bracket $\br{e_1}{e_2} = e_3$. Clearly the
center is $\laz = \span{e_3}$. Consider the Lorentzian inner product on $\lah_3$
given on this basis by 
$$
-\inp{e_1}{e_1} = \inp{e_2}{e_2} = \inp{e_3}{e_3} = 1.
$$
Now let $(H_3,\inpe)$ be the 3-dimensional Heisenberg group with the left-invariant
Lorentz metric induced by this inner product on $\lah_3$.
The $j$-maps on this group satisfy
$$
j(z)^2 = \norm{z}_\laz^2 \id|_\lav,
$$
so this group is not of pseudo-$H$-type. However, taking $\vp = -\norme_\laz^2$ one
sees that it is of modified $H$-type. Therefore the class of modified $H$-type metrics
is strictly larger than the class of pseudo-$H$-type metrics.
\end{example}

\begin{example}
In\label{Rgheis} the Riemannian case, the \emph{generalized Heisenberg groups}
satisfy 
$$
j(z)^2 = -4c\norm{z}_\laz^2 \id|_\lav
$$
for all $z \in \laz$, $c \in (0,\infty)$. See \cite{Prip, E1}, for example. These groups
are of modified $H$-type with
$\vp = 4c\norme_\laz^2$, but they are only of $H$-type when $c = \tfrac{1}{4}$. Thus
the class of modified $H$-type metrics is strictly larger than the class of $H$-type
metrics.

This example may be extended straightforwardly to pseudo-Riemannian groups
to obtain the pseudo-Riemannian generalized Heisenberg groups.
\end{example}

Each of the previous examples constructs modified $H$-type metrics on groups that
are already known to admit $H$-type or pseudo-$H$-type metrics. The next examples
show that the class of groups admitting modified $H$-type metrics is strictly larger
than those admitting (pseudo-) $H$-type metrics.

Consider the 4-dimensional group obtained by a trivial central extension of the
Heisenberg group, $N = H_3(\R) \times \R$. The Lie algebra $\n = \lah_3 \ds \R$
has basis $\{e_0,e_1,e_2,e_3\}$ with nontrivial bracket $\br{e_1}{e_2} = e_3$. Thus
$\laz = \span{e_0,e_3}$ and $\lav = \span{e_1,e_2}$. The following examples show
that there are modified $H$-type metrics on $N$ of each possible signature.

Moreover, this group does not admit an $H$-type or pseudo-$H$-type metric.
Indeed, it does not admit an $H$-type metric because it is singular, and it does
not admit a pseudo-$H$-type metric because the minimum dimension of a pseudo-%
$H$-type group with 2-dimensional center is six; \emph{viz.} \cite{Ci}.

\begin{example}
Consider\label{teh0} the Riemannian metric $g_0$ on $N = H_3(\R) \times \R$ making
the basis $\{e_0,e_1,e_2,e_3\}$ orthonormal. That is,
$$
\inp{e_i}{e_j} = \begin{cases}
1 & i = j \\
0 & i \neq j
\end{cases}.
$$
The $j$-maps on this group are given by
$$
j(e_0) = \begin{pmatrix}
0 & 0 \\
0 & 0
\end{pmatrix},\ \ \text{and}\ \ j(e_3) = \begin{pmatrix}
0 & -1 \\
1 & 0 
\end{pmatrix},
$$
so that their squares satisfy $j(e_0)^2 = 0$ and $j(e_3)^2 = -\id|_\lav$. The maps
$j(e_0 + e_3) = j(e_3)$ and $j(e_0 - e_3) = - j(e_3)$, so that the $j(e_0 + e_3)^2 
= j(e_0 - e_3)^2 = - \id|_\lav$. Let $\vp_0$ be the quadratic form on $\laz$
satisfying
$$
\vp_0(e_0) = 0,\ \ \vp_0(e_3) = \vp_0(e_0 + e_3) = \vp_0(e_0 - e_3) = 1.
$$
The induced symmetric bilinear form on $\laz$ satisfies 
$$
\scp{e_0}{e_3}_{\vp_0} = \tfrac{1}{4}\left(\vp_0(e_0 + e_3) - \vp_0(e_0 - e_3)\right) 
= \tfrac{1}{4} \left( 1 - 1 \right) = 0,
$$
so that its matrix representation is
$$
\Phi_0 = \begin{pmatrix}
0 & 0 \\
0 & 1
\end{pmatrix}.
$$
This bilinear form is exactly what one should expect since $N$ is a trivial extension
of $H_3(\R)$.
We see that $j(z)^2 = -\vp_0(z)\,\id|_\lav$ for every $z \in \laz$. Therefore 
$(N,g_0,\vp_0)$ is a modified $H$-type group with degenerate quadratic form $\vp_0$.
\end{example}

\begin{example}
Now\label{teh1} consider the left-invariant Lorentzian metric $g_1$ on $N$ given by
$$
\inp{e_0}{e_3} = \inp{e_1}{e_1} = \inp{e_2}{e_2} = 1,
$$
with all other combinations zero. The $j$-maps for this group are
$$
j(e_0) = \begin{pmatrix}
0 & -1 \\
1 & 0
\end{pmatrix}, \text{\ \ and\ \ } j(e_3) = \begin{pmatrix}
0 & 0 \\
0 & 0
\end{pmatrix},
$$
so that $j(e_0)^2 = -\id|_\lav$ and $j(e_3)^2 = 0$. Taking a pseudo-orthonormal
basis of $\laz$,
$$
u_1 = \tfrac{1}{\sqrt{2}} (e_0 + e_3), \ \ u_2 = \tfrac{1}{\sqrt{2}} (e_0 - e_3),
$$
we obtain $j(u_1) = j(u_2) = \tfrac{1}{\sqrt{2}} j(e_0)$, so that $j(u_1)^2
= j(u_2)^2 = -\tfrac{1}{2}\,\id|_\lav$. Moreover, $u_1 + u_2 = \sqrt{2}e_0$ and
$u_1 - u_2 = -\sqrt{2}e_3$, so that $j(u_1 + u_2)^2 = -2\, \id|_\lav$ and 
$j(u_1 - u_2)^2 = 0$. Let $\vp_1$ be the quadratic form on $\laz$ satisfying 
$$
\vp_1(u_1) = \vp_1 (u_2) =  \tfrac{1}{2},\ \ \vp_1(u_1 + u_2) = 2,\ \ \text{and}\ \ 
\vp(u_1 - u_2) = 0.
$$
The induced symmetric bilinear form on $\laz$ satisfies
$$
\scp{u_1}{u_2}_{\vp_1} = \tfrac{1}{4}\left(\vp_1(u_1 + u_2) - \vp_1(u_1 - u_2)\right)
= \tfrac{1}{4}(2 - 0) = \tfrac{1}{2}.
$$
Therefore its matrix representation is
$$
\Phi_1 = \begin{pmatrix}
\tfrac{1}{2} & \tfrac{1}{2} \\[0.25 ex]
\tfrac{1}{2} & \tfrac{1}{2}
\end{pmatrix},
$$
and $\vp_1$ is degenerate. However, $j(z)^2 = -\vp_1(z)\, \id|_\lav$ for
all $z \in \laz$, so $(N,g_1,\vp_1)$ is indeed a modified $H$-type group.
\end{example}

\begin{example}
Next\label{teh2} consider the left-invariant neutral metric $g_2$ on $N$ given
on the Lie algebra $\n$ by
$$
\inp{e_0}{e_3} = \inp{e_1}{e_2} = 1,
$$
with all other combinations zero. Before computing the $j$-maps, change the
basis of $\n$ to
\begin{eqnarray*}
u_1 & = & \tfrac{1}{\sqrt{2}} \left(e_0 + e_3\right), \\
u_2 & = & \tfrac{1}{\sqrt{2}} \left(e_0 - e_3\right), \\
v_1 & = & \tfrac{1}{\sqrt{2}} \left(e_1 + e_2\right), \\
v_2 & = & \tfrac{1}{\sqrt{2}} \left(e_1 - e_2\right).
\end{eqnarray*}
The basis $\{u_1,u_2,v_1,v_2\}$ is pseudo-orthonormal with
\begin{eqnarray*}
\inp{u_1}{u_1} & = & - \inp{u_2}{u_2} = 1, \text{\ and} \\
\inp{v_1}{v_1} & = & - \inp{v_2}{v_2} = 1.
\end{eqnarray*}
The Lie bracket is now given by $\br{v_1}{v_2} = \tfrac{1}{\sqrt{2}}\left(u_1 - u_2
\right)$,  $\laz = \span{u_1,u_2}$, and $\lav = \span{v_1,v_2}$. The $j$-maps
are given with respect to these bases by
$$
j(u_1) = j(u_2) = \frac{1}{\sqrt{2}} \begin{pmatrix}
0 & -1 \\
-1 & 0
\end{pmatrix},
$$
so that
$$
j(u_1 + u_2) = \sqrt{2} \begin{pmatrix}
0 & -1 \\
-1 & 0
\end{pmatrix}.
$$
We see that $j(u_1)^2 = j(u_2)^2 = \tfrac{1}{2}\, \id|_\lav$, and $j(u_1 + u_2)^2 
= 2\,\id|_\lav$. Let $\vp_2$ be the quadratic form on
$\laz$ satisfying
$$
\vp_2(u_1) = \vp_2(u_2) = -\tfrac{1}{2}, \ \ \text{and}\ \ \vp_2(u_1 + u_2) = -2.
$$
Then the associated bilinear form $\scp{\ }{\,}_{\vp_2}$ satisfies 
$$
\scp{u_1}{u_2}_{\vp_2} = \tfrac{1}{2} \left( \vp_2(u_1 + u_2) - \vp_2(u_1) - 
\vp_2(u_2)\right) = \tfrac{1}{2}\left( -2 + \tfrac{1}{2} + \tfrac{1}{2}\right) = 
-\tfrac{1}{2},
$$
and its matrix representation is
$$
\Phi_2 = \begin{pmatrix}
-\tfrac{1}{2} & -\tfrac{1}{2} \\[0.25 ex]
-\tfrac{1}{2} & -\tfrac{1}{2}
\end{pmatrix}.
$$
We see that $j(z)^2 = -\vp_2(z)\,\id|_\lav$ for all $z \in \laz$, so that
$(N,g_2,\vp_2)$ is a modified $H$-type group.
\end{example}

\begin{example}
Consider\label{teh3} the left-invariant metric $g_3$ given on $\n$ by
$$
\inp{e_0}{e_0} = \inp{e_3}{e_3} = \inp{e_1}{e_2} = 1.
$$ 
The $j$-maps with respect to the basis $\{e_1,e_2\}$ are
$$
j(e_0) = \begin{pmatrix}
0 & 0 \\
0 & 0
\end{pmatrix}\ \ \text{and}\ \ j(e_3) = \begin{pmatrix}
1 & 0 \\
0 & -1
\end{pmatrix},
$$
so that $j(e_0)^2 = 0$ and $j(e_3)^2 = \id|_\lav$. Notice that $j(e_0 + e_3) = 
j(e_3)$, so that $j(e_0 + e_3)^2 = \id|_\lav$ as well. Let $\vp_3$ be the quadratic
form on $\laz$ satisfying
$$
\vp_3(e_0) = 0,\ \ \vp_3(e_3) = \vp_3(e_0 + e_3) = -1.
$$
The bilinear form $\scpe_{\vp_3}$ satisfies
$$
\scp{e_0}{e_3}_{\vp_3} = \tfrac{1}{2}\left(\vp_3(e_0 + e_3) - \vp_3(e_0) - 
\vp_3(e_3)\right) = \tfrac{1}{2}(-1 - 0 + 1) = 0,
$$
so its matrix representation is 
$$
\Phi_3 = \begin{pmatrix}
0 & 0 \\
0 & -1
\end{pmatrix}.
$$
Therefore $(N,g_3,\vp_3)$ is a modified $H$-type group.
\end{example}

\section{Connection and Curvatures}

In this section we give all relevant curvature formulas for modified $H$-type groups.
Let $(N,\inpe,\vp)$ be a modified $H$-type Lie group. Let $\scppe$ be the symmetric
bilinear form
on $\laz$ obtained by polarizing the quadratic form $\vp$. Then for all $e,e' \in 
\lav$ and all $z, z' \in \laz$, 
\begin{eqnarray}
\inp{j(z)e}{j(z)e'}_\lav\label{jmpol1} & = & \vp(z) \inp{e}{e'}_\lav, \\
\inp{j(z)e}{j(z')e}_\lav\label{jmpol2} & = & \scpp{z}{z'}\inp{e}{e}_\lav,
\text{\ and} \\
j(z)j(z') + j(z')j(z)\label{jmpol3} & = & -2\scpp{z}{z'}\id|_\lav.
\end{eqnarray}

These identities as well as the following formulas for the Levi-Civita connection
and curvatures may be deduced from \cite{CP4}. We prove only the curvature formulas
below that are unique to modified $H$-type groups.

\begin{theorem}
Let\label{lcc} $(N,\inpe)$ be a 2-step nilpotent Lie group with left-invariant
pseudo-Riemannian metric making the center nondegenerate. For all $e, e' \in \lav$
and $z, z' \in \laz$, the Levi-Civita connection on $N$ is given by
\begin{eqnarray}
\del{z}z'\label{lcc1} & = & 0, \\
\del{z}e\label{lcc2} = \del{e}z & = & -\tfrac{1}{2} j(z)e, \\
\del{e}e'\label{lcc3} & = & \tfrac{1}{2} \br{e}{e'}.
\end{eqnarray}
\end{theorem}

\begin{theorem}
Let\label{riem} $(N,\inpe)$ be a 2-step nilpotent Lie group with left-invariant
pseudo-Riemannian metric making the center nondegenerate. For all $e, e', e'' \in 
\lav$ and $z, z', z'' \in \laz$, the Riemann curvature tensor on $N$ is given by
\begin{eqnarray}
R(z,z')z''\label{riem1} & = & 0, \\
R(z,z')e\label{riem2} & = & \tfrac{1}{4}\br{z}{z'}e, \\
R(z,e)z'\label{riem3} & = & \tfrac{1}{4}j(z)j(z')e, \\
R(z,e)e' \label{riem4} & = & \tfrac{1}{4} \br{e}{j(z)e'}, \\
R(e,e')z\label{riem5} & = & -\tfrac{1}{4}\left(\br{e}{j(z)e'} + \br{j(z)e}{e'}
\right), \\
R(e,e')e''\label{riem6} & = & \tfrac{1}{4}\left( j\left(\br{e}{e''}\right)e' -
j\left(\br{e'}{e''}\right)e \right)+ \tfrac{1}{2}j\left(\br{e}{e'}\right)e''.
\end{eqnarray}
\end{theorem}


\pagebreak

\begin{theorem}
Let $(N,\inpe,\vp)$\label{sec} be a modified $H$-type group, and suppose that
$e, e' \in \lav$, $z, z' \in \laz$ are pseudo-orthonormal. The sectional
curvature is given by
\begin{eqnarray}
K(z,z')\label{sec1} & = & 0, \\
K(z,e)\label{sec2} & = & \tfrac{1}{4}\ve_{z}\,\vp(z), \\
K(e,e')\label{sec3} & = & -\tfrac{3}{4} \ve_{e}\, \ve_{e'}\,
\norm{\br{e}{e'}}_\laz^2,
\end{eqnarray}
where $\ve_\bullet = \norm{\bullet}^2 = \pm 1$.
\end{theorem}

\begin{proof}
Only equation \eqref{sec2} differs from the usual ones in \cite{CP4}. Let $z \in
 \laz$, $e \in \lav$, such that $\ve_z = \norm{z}_\laz^2 = \pm 1$ and $\ve_e =
\norm{e}_\lav^2 = \pm 1$. Then by \eqref{jmpol1},
\begin{eqnarray*}
K(z,e) 	& = & \tfrac{1}{4} \ve_z\, \ve_e\, \inp{j(z)e}{j(z)e}_\lav \\
		& = & \tfrac{1}{4} \ve_z\, \ve_e\, \vp(z)\, \norm{e}_\lav^2 \\
		& = & \tfrac{1}{4} \ve_z\, \vp(z).
\end{eqnarray*}
\end{proof}

Let $\{z_1,\hdots,z_p\}$ be a pseudo-orthonormal basis for the center $(\laz,\inpe_\laz)$,
and let $\{e_1,\hdots,e_m\}$ be a pseudo-orthonormal basis for $(\lav,\inpe_\lav)$. Define
\begin{equation}
\xi\label{xi} := \sum_{k = 1}^p \norm{z_k}_\laz^2 \, \vp(z_k).
\end{equation}
This number does not depend on the choice of basis of $\laz$ as it is the trace
of the operator $\escr{K} :z \mapsto \norm{z}_\laz^2 \vp(z)$.

\begin{theorem}
Let\label{ric} $(N,\inpe,\vp)$ be a modified $H$-type group, and let $e, e' \in 
\lav$, $z, z' \in \laz$. The Ricci curvature is given by
\begin{eqnarray}
\ric{e}{z}\label{ric1} & = & 0, \\
\ric{e}{e'}\label{ric2} & = & -\tfrac{\xi}{2} \inp{e}{e'}_\lav, \\
\ric{z}{z'}\label{ric3} & = & \tfrac{m}{4} \scpp{z}{z'}.
\end{eqnarray}
\end{theorem}

\begin{proof}
Equation \eqref{ric1} is true for all pseudo-Riemannian 2-step nilpotent groups with
nondegenerate centers.
Using the pseudo-orthonormal bases of \eqref{xi} and letting $z, z' \in \laz$,
$e, e' \in \lav$, we compute
\begin{eqnarray*}
\ric{e}{e'}	& = & -\frac{1}{2}\sum_{k = 1}^p \norm{z_k}_\laz^2 \inp{j(z_k)e}
{j(z_k)e'}_\lav \\
			& = & -\frac{1}{2}\sum_{k = 1}^p \norm{z_k}_\laz^2 \vp(z_k) \inp{e}
{e'}_\lav \\
			& = & -\frac{1}{2} \inp{e}{e'}_\lav \sum_{k = 1}^p \norm{z_k}_\laz^2
			\vp(z_k) \\
			& = & -\frac{\xi}{2} \inp{e}{e'}_\lav,
\end{eqnarray*}
and
\begin{eqnarray*}
\ric{z}{z'} 	& = & \frac{1}{4} \sum_{i = 1}^m \norm{e_i}_\lav^2 
			\inp{j(z)e_i}{j(z')e_i}_\lav \\
			& = & \frac{1}{4} \sum_{i = 1}^m \norm{e_i}_\lav^2 \scpp{z}{z'}
			\inp{e_i}{e_i}_\lav \\
			& = & \frac{m}{4} \scpp{z}{z'}.
\end{eqnarray*}
\end{proof}

The Ricci operator $\rc: \n \to \n$ is defined by
$$
\inp{\rc\, x	}{y} = \ric{x}{y}
$$
for all $x, y \in \n$. By \eqref{ric1}, the Ricci operator preserves the
splitting $\n = \laz \ds \lav$. By \eqref{ric3}, the Ricci operator restricted to
$\laz$ is determined by
\begin{equation}
\inp{\rc\, z}{z'}_\laz\label{rop*} = \tfrac{m}{4} \scpp{z}{z'}.
\end{equation}
Denote by $\id|^\dagger_\vp$, the adjoint of the identity map on $\laz$ with
respect to the metrics $\scppe$ and $\inpe_\laz$.
This adjoint is defined by the equation
\begin{equation}
\inp{\id|^\dagger_\vp\, z}{z'}_\laz\label{idad} = \scpp{z}{\id\, z'}.
\end{equation}
Comparing \eqref{rop*} and \eqref{idad}, $\rc|_\laz$ is proportional to
$\id|_\vp^\dagger$. 
We deduce the following theorem regarding the Ricci operator of modified $H$-type
groups.

\begin{theorem}
Let\label{rop} $(N,\inpe,\vp)$ be a modified $H$-type group. The Ricci operator
preserves the splitting $\n = \laz \ds \lav$, and is given on each factor by
\begin{eqnarray}
\rc|_\lav\label{rop1} & = & -\tfrac{\xi}{2}\, \id|_\lav, \\
\rc|_\laz\label{rop2} & = & \tfrac{m}{4}\, \id|^\dagger_\vp.
\end{eqnarray}
\end{theorem}

\begin{remark}
In order to better understand the operator $\id|_\vp^\dagger$ it is
useful to work with matrix representations. Let $E$ be the matrix representing
$\inpe_\laz$ with respect to the basis $\{z_1,\hdots,z_p\}$, so that 
$E = \text{diag}\{\ve_1,\hdots,\ve_p\}$. Let $\Phi$ be the matrix representation
of $\scppe$ with respect to the same basis. In general $\Phi$ is symmetric, but
may be singular. 

Equation \eqref{rop*} may be written in matrix form as 
$(\rc|_\laz\, z)^T E z' = \tfrac{m}{4} z^T \Phi z'$. Taking advantage of the facts
that $E^T = E^{-1} = E$ and $\Phi^T = \Phi$, basic matrix manipulations yield
\begin{eqnarray*}
(\rc|_\laz\, z)^T E & = & \tfrac{m}{4}z^T \Phi \\
(\rc|_\laz\, z)^T & = & \tfrac{m}{4} z^T \Phi E^{-1} \\
(\rc|_\laz\, z)^T & = & \big[\tfrac{m}{4}(\Phi E^{-1})^T z\big]^T
\end{eqnarray*}
This implies that $\rc|_\laz = \tfrac{m}{4}(E^{-1})^T\Phi^T =
\tfrac{m}{4}E^{-1}\Phi$. That is,
the matrix representation of $\rc|_\laz$ with respect to the orthonormal basis 
$\{z_1,\hdots,z_p\}$ is 
\begin{equation}
\rc|_\laz\label{rvzmat} = \tfrac{m}{4}\, E^{-1} \Phi.
\end{equation}
Notice that $\xi := \sum_i\ve_i\,\vp(z_i)$ is the trace of the matrix $E^{-1} \Phi$.

If $N$ is of (pseudo-) $H$-type, then $\Phi = E$, $E^{-1}\Phi = I_p$, and we recover
all of the previously known results about the Ricci operator.
\end{remark}

\begin{remark}
Clearly $\rc(z) = 0$ implies that either $z \in \ker \Phi$ or $\Phi z$ is null
with respect to the metric on $\laz$. This fact will become useful later.
\end{remark}

Scalar curvature is the trace of the Ricci curvature operator.

\begin{theorem}
Let\label{scal} $(N,\inpe,\vp)$ be a modified $H$-type group. Since the metric
$\inpe$ is left-invariant, the scalar curvature is constant on $N$, and is
given by
\begin{equation}
S\label{scal1} = - \frac{m\xi}{4}.
\end{equation}
\end{theorem}

\begin{proof}
Using the pseudo-orthonormal bases of \eqref{xi}, we compute
\begin{eqnarray*}
S 	& = & -\frac{1}{4} \sum_{i = 1}^m \sum_{k = 1}^p \norm{z_k}^2_\laz
	\norm{e_i}_\lav^2 \inp{j(z_k)e_i}{j(z_k)e_i}_\lav \\
	& = & -\frac{1}{4} \sum_{i = 1}^m \sum_{k = 1}^p \norm{z_k}_\laz^2 \vp(z_k) \\
	& = & -\frac{m\xi}{4}.
\end{eqnarray*}
\end{proof}

\subsection{\normalsize \textit{Examples}}

We now return to the examples of Section \ref{de}. We compute the constant
$\xi$ and the matrix representation of the Ricci operator, and compute its
eigenvalues for each group.

\begin{example}
Consider\label{teh0ex1} the group $(N,g_0,\vp_0)$ of Example \ref{teh0}.
The matrices representing $g_0$ and $\scp{\ }{\,}_{\vp_0}$ on $\laz$ are
$$
E = \begin{pmatrix}
1 & 0 \\
0 & 1
\end{pmatrix},\ \ \text{and}\ \ 
\Phi_0 = \begin{pmatrix}
0 & 0 \\
0 & 1
\end{pmatrix}.
$$
Thus $E^{-1}\Phi_0 = \Phi_0$, $\xi = \tr{\Phi_0} = 1$, and the scalar curvature of
$(N,g_0)$ is $S = -\tfrac{1}{2}$. The Ricci operator is given by
\begin{eqnarray*}
\rc|_\lav & = & -\frac{1}{2}\begin{pmatrix}
1 & 0 \\
0 & 1
\end{pmatrix}, \text{\ and} \\[0.5 ex]
\rc|_\laz & = & \frac{1}{2} \begin{pmatrix}
0 & 0 \\
0 & 1
\end{pmatrix}.
\end{eqnarray*}
It has eigenvalues $-\tfrac{1}{2}$ on $\lav$ and $0, \tfrac{1}{2}$ on $\laz$.
\end{example}

\begin{example}
Consider\label{teh1ex1} the modified $H$-type group $(N,g_1,\vp_1)$ of Example
\ref{teh1}. The matrix representation of $g_1$ on $\laz$ with respect to the 
ordered basis $\{u_1,u_2\}$ is
$$
E = \begin{pmatrix}
1 & 0 \\
0 & -1
\end{pmatrix},
$$
while the matrix representation of $\scp{\ }{\,}_{\vp_1}$ with respect to the 
same basis is 
$$
\Phi_1 = \begin{pmatrix}
\tfrac{1}{2} & \tfrac{1}{2} \\[0.25 ex]
\tfrac{1}{2} & \tfrac{1}{2}
\end{pmatrix}.
$$
The matrix $E^{-1}\Phi_1$ is then
$$
E^{-1}\Phi_1 = \frac{1}{2}\begin{pmatrix}
1 & 1 \\
-1 & -1
\end{pmatrix}.
$$
Clearly $\xi = \tr{E^{-1}\Phi_1} = 0$, so that $(N,g_1)$ is scalar flat. This also implies
that the Ricci operator restricted to $\lav$ is zero. Moreover, the
Ricci operator on $\laz$, given by the matrix
$$
\rc|_\laz = \frac{1}{4} \begin{pmatrix}
1 & 1 \\
-1 & -1
\end{pmatrix},
$$
is 2-step nilpotent. Therefore the Ricci operator on the entire Lie algebra $\n$ has
only one eigenvalue, $\lambda = 0$.

The space $(N,g_1)$ was also found to be scalar flat but not Ricci flat in \cite{BO}.
\end{example}

\begin{example}
Next\label{teh2ex1} consider the group $(N,g_2,\vp_2)$ of Example \ref{teh2}. The
matrix representations of $g_2$ and $\scp{\ }{\,}_{\vp_2}$ with respect to the basis
$\{u_1,u_2\}$ are
$$
E = \begin{pmatrix}
1 & 0 \\
0 & -1
\end{pmatrix},\ \ \text{and\ \ } \Phi_2 = -\frac{1}{2} \begin{pmatrix}
1 & 1 \\
1 & 1
\end{pmatrix}.
$$
The matrix $E^{-1}\Phi_2$ is then
$$
E^{-1}\Phi_2 = -\frac{1}{2} \begin{pmatrix}
1 & 1 \\
-1 & -1
\end{pmatrix}.
$$
As in Example \ref{teh1ex1}, we see that $\xi = \tr{E^{-1}\Phi_2} = 0$, so the space
$(N,g_2)$ has zero scalar curvature. This also implies that the Ricci operator
vanishes on $\lav$. Again in analogy with Example \ref{teh1ex1}, the Ricci operator
restricted to $\laz$ is 2-step nilpotent,
$$
\rc|_\laz = -\frac{1}{4} \begin{pmatrix}
1 & 1 \\
-1 & -1
\end{pmatrix}. 
$$
Therefore $\rc$ has only one eigenvalue, $\lambda = 0$.
\end{example}

\begin{example}
Consider\label{teh3ex1} the group $(N,g_3,\vp_3)$ of Example \ref{teh3}. The matrices
for $g_3$ and $\scpe_{\vp_3}$ on $\laz$ are given by
$$
E = \begin{pmatrix}
1 & 0 \\
0 & 1
\end{pmatrix},\ \ \text{and}\ \ \Phi_3 = \begin{pmatrix}
0 & 0 \\
0 & -1
\end{pmatrix}.
$$
Thus $E^{-1}\Phi_3 = \Phi_3$, and $\xi = \tr{\Phi_3} = -1$. Therefore the scalar curvature
of $(N,g_3)$ is $S = +\tfrac{1}{2}$. The Ricci operator on $\n$ is given by
\begin{eqnarray*}
\rc|_\lav & = & \frac{1}{2} \begin{pmatrix}
1 & 0 \\
0 & 1
\end{pmatrix},\ \text{and} \\[0.25 ex]
\rc|_\laz & = & \frac{1}{2} \begin{pmatrix}
0 & 0 \\
0 & -1
\end{pmatrix}.
\end{eqnarray*}
The eigenvalues of $\rc$ are then $\tfrac{1}{2}$ on $\lav$ and $0, -\tfrac{1}{2}$
on $\laz$.
\end{example}

\section{Geometry of Modified $H$-Type Groups}

In this section we investigate the geometric consequences of Definition \ref{gh}.
Most of these properties can be stated in terms of the Ricci operator of 
Theorem \ref{rop}.

\subsection{\normalsize \textit{Isometry groups}}

Let $\n^\C = \laz^\C \ds \lav^\C$ be the complexification of the Lie algebra $\n$.
Since the Ricci operator of Theorem \ref{rop} respects the splitting $\n = \laz
\ds \lav$, then its complexification respects the induced splitting of $\n^\C$.

Recall the following lemma of \cite{BO}.

\begin{lemma}[\cite{BO}, Lemma 2]
Let\label{BOL2} $(N,\inpe)$ be a 2-step nilpotent Lie group such that $\inpe$ is a 
pseudo-Riemannian left-invariant metric for which the center is nondegenerate.
Assume
\begin{eqnarray}
\lav^\C & = & V_{\lambda_1} \ds \cdots \ds V_{\lambda_j}, \notag \\
\laz^\C & = & V_{\lambda_{j+1}} \ds \cdots \ds V_{\lambda_s},
\end{eqnarray}
for the different eigenvalues $\lambda_1, \hdots, \lambda_s$ of the Ricci
operator $\rc$, where $V_{\lambda_i}$ are the eigenspaces corresponding to
$\lambda_i$. Then every isometry of $N$ preserves the splitting $TN = \lav N
\ds \laz N$; that is, $\Isp(N) = \Iso(N)$. \eop
\end{lemma}
Since $\Isa(N) = \Isp(N)$ for all groups with nondegenerate centers, this actually
implies that $\Isa(N) = \Isp(N) = \Iso(N)$ for all such groups satisfying the
hypotheses of the lemma.

\medskip

For modified $H$-type groups, the restriction of $\rc$ to $\lav$ has only one
eigenvalue, $-\tfrac{\xi}{2}$.
As long as $-\tfrac{\xi}{2}$ is not also an eigenvalue of $\rc|_\laz$, then the
eigenspace decomposition of $\n$ with respect to the Ricci operator respects 
the splitting $\n = \laz \ds \lav$. That is, if $V_{\lambda_i}$ are the eigenspaces
for the different eigenvalues $\lambda_i$ of $\rc|_\laz$, then
\begin{eqnarray}
\laz^\C & = & V_{\lambda_1} \ds \cdots \ds V_{\lambda_s}, \text{\ and} \notag \\
\lav^\C\label{evd} & = & V_{-\frac{\xi}{2}}.
\end{eqnarray}

Invoking Lemma \ref{BOL2} and recalling that modified $H$-type groups have
nondegenerate centers by definition, we obtain the following theorem about the
isometry groups of modified $H$-type groups.

\begin{theorem}
Let $(N,\inpe,\vp)$ be a modified $H$-type group. If $-\tfrac{\xi}{2}$ is
\emph{not} an eigenvalue of $\rc|_\laz$, then the following isometry groups
coincide: $\Isa(N) = \Isp(N) = \Iso(N)$. \eop
\end{theorem}

\begin{remark}
The groups $(N,g_1)$ and $(N,g_2)$ of Examples \ref{teh1}, \ref{teh1ex1}, \ref{teh2},
and \ref{teh2ex1} do not satisfy the hypothesis of the theorem; \emph{viz.}
\cite{BO}. 
\end{remark}

\subsection{\normalsize \textit{Totally geodesic subgroups}}

Let $z \in \laz$ and $x \in \lav$ be non-zero vectors. In $H$-type 
groups, the space $\n' := \span{z,x,j(z)x}$ is always a totally geodesic subalgebra of
$\n$, making $N' := \exp{\n'}$ a totally geodesic subgroup of $N$.
Modified $H$-type groups have the following analogue.

\begin{theorem}
Let\label{tgsg} $(N,\inpe,\vp)$ be a modified $H$-type Lie group. Let $z \in \laz$,
$x \in \lav$ be non-zero vectors with $\norm{x}_\lav^2 \neq 0$, and consider the
subspace
$\n' = \span{z,x,j(z)x}$. The submanifold $N' = \exp{\n'}$ is a totally geodesic
subgroup of $N$ if and only if $z$ is not in the kernel of $\rc|_\laz$ but is an
eigenvector of $\rc|_\laz$.
\end{theorem}

\begin{proof}
The Levi-Civita connection on $\n'$ is given by
\begin{eqnarray}
\del{z}x = \del{x}z & = & -\tfrac{1}{2} j(z)x, \notag \\
\del{z}j(z)x & = & \tfrac{1}{2} \vp(z)\, x, \text{\ and}\label{tgsc} \\ 
\del{x}j(z)x & = & \tfrac{1}{2} \br{x}{j(z)x}. \notag 
\end{eqnarray}
Thus, we must show that $z' := \br{x}{j(z)x}$ is proportional to $z$ if and only if
$z$ is an eigenvector of $\rc|_\laz$. For $a \in \laz$, we compute
\begin{eqnarray}
\inp{\br{x}{j(z)x}}{a}_\laz\label{tgsgp1} 	& = & \inp{\ad{x}j(z)x}{a}_\laz
\notag \\
						& = & \inp{j(z)x}{j(a)x}_\laz \\
						& = & \norm{x}_\lav^2\, \scpp{z}{a}. \notag
\end{eqnarray}
Therefore, $z' = \norm{x}_\lav^2\, \id|_\vp^\dagger \, z$.
By \eqref{rop2}, $z'= \norm{x}_\lav^2\, \tfrac{4}{m}\,
\rc\, z$, and $z'$ is proportional to $z$ if and only if $z$ is an
eigenvector of $\rc$.
\end{proof}

\begin{corollary}
Let\label{tgsgc1} $(N,\inpe,\vp)$ be a modified $H$-type Lie group. Let $z \in \laz$
and $x \in \lav$ be non-zero vectors. Suppose that either $z \in \ker \rc|_\laz$ or
$\norm{x}_\lav^2 = 0$ (or both), and consider the subspace $\n' = \span{z,x,j(z)x}$.
The submanifold $N' = \exp\n'$ is totally geodesic, but not a subgroup of $N$. 
\end{corollary}

\begin{proof}
By \eqref{tgsc}, $\del{z}x, \del{x}z,$ and $\del{z}j(z)x$ are all in $\n'$. 
By \eqref{tgsgp1}, $\del{x}j(z)x = 0$. Therefore $N'$ is totally geodesic in $N$.
However, $\n'$ is not a subalgebra since $\br{x}{j(z)x} = 0$.
\end{proof}



\subsection{\normalsize \textit{Geodesics}}

In this section we give explicit formulas for the geodesics of a modified $H$-type
group in terms of the $j$-maps, quadratic form $\vp$, and the Ricci operator on
$\laz$. We follow the calculations of \cite{CP4}, making particular use of one
result that is stated here as a lemma. First, some preliminaries. 

Suppose $N$ is a connected, simply connected 2-step nilpotent Lie group.
Let $I$ be a real interval containing zero, and let $\g :I \to N$ be a geodesic
such that $\g(0) = 1 \in N$ and $\dg(0) = z_0 + x_0 \in \laz \ds \lav = \n$. 
Since the exponential map is a diffeomorphism for simply connected
nilpotent Lie groups, $\g(t) = \exp(z(t) + x(t))$.

Let $J = j(z_0)$ denote the skew adjoint transformation of $\lav$ determined by the
initial condition, and let $\lav^1 = \ker J$. The skew-adjointness of $J$ implies
that $\lav = \lav^1 \ds \lav^2$ is an orthogonal direct sum, where $\lav^2 =
\lav/\lav^1$. Notice that $J$ is invertible on $\lav^2$. Decompose $x_0 = x_1 +
x_2$, with $x_i \in \lav^i$.

Now let $\{\theta_1,\hdots,\theta_k\}$ be the distinct nonzero eigenvalues of
$J^2$. Decompose $\lav^2$ as the orthogonal direct sum $\bigoplus_{j = 1}^k
\lal{w}_j$, where $J$ leaves each $\lal{w}_j$ invariant and $J^2|_{\lal{w}_j} = 
\theta_j \id|_\lav$. Write $x_2 = \sum_{j = 1}^k w_j$ with each $w_j \in \lal{w}_j$.

Under these assumptions, we state the aforementioned result of \cite{CP4}.

\begin{lemma}[\cite{CP4}, Prop 4.15]
If\label{CPLem} $N$ is a 2-step nilpotent Lie group with nondegenerate center, and
$J^2$ diagonalizes, then
\begin{eqnarray}
z(t)\label{geolz} & = & tz_1(t) + z_2(t), \\
x(t)\label{geolx} & = & tx_1(t) + (e^{tJ} - I)J^{-1}x_2, 
\end{eqnarray}
where
\begin{eqnarray*}
z_1(t) & = & z_0 + \tfrac{1}{2}\br{x_1}{(e^{tJ} + I)J^{-1}x_2} + \tfrac{1}{2}
       \sum_{j = 1}^k\br{J^{-1}w_j}{w_j}, \\
z_2(t) & = & \br{x_1}{(I - e^{tJ})(J^{-2}x_2)} + \tfrac{1}{2}\br{e^{tJ}J^{-1}x_2}
       {J^{-1}x_2} \\
       &   & - \tfrac{1}{2} \sum_{i \neq j} \frac{1}{\theta_j - \theta_i} \left(
       \br{e^{tJ}Jw_i}{e^{tJ}J^{-1}w_j} - \br{e^{tJ}w_i}{e^{tJ}w_j}\right) \\
       &   & + \tfrac{1}{2}\sum_{i \neq j} \frac{1}{\theta_j - \theta_i} \left(
       \br{Jw_i}{J^{-1}w_j} - \br{w_i}{w_j}\right).
\end{eqnarray*}
\end{lemma}

For modified $H$-type groups, the situation is simplified considerably.
The map $J^2$ is diagonalizable for all $z_0 \in \laz$ and has only one eigenvalue,
$-\vp(z_0)$. This eliminates the need for any $\lal{w}$-decomposition, thus
eliminating many of the terms in $z_1$ and $z_2$ of the 
lemma. Further, the maps $J^{-1}$ and $e^{tJ}$ may each be written simply in 
terms of $\vp(z_0)$ and $j(z_0)$. The following lemma is straightforward to prove.

\begin{lemma}
Let\label{Jcalc} $(N,\inpe,\vp)$ be a modified $H$-type Lie group. Let $z_0 \in 
\laz$ such
that $\vp_0 := \vp(z_0) \neq 0$, and denote by $J$ the transformation $j(z_0) \in
\End(\lav)$. Then
\begin{eqnarray*}
J^{-1} & = & -\frac{1}{\vp_0} J, \text{\ and} \\
e^{tJ} & = & \begin{cases}
\cos\left(t\sqrt{\vp_0}\right) I +
\frac{1}{\sqrt{\vp_0}}\sin\left(t\sqrt{\vp_0}\right)J & \text{if\ }
\vp_0 > 0,\\[1 ex]
\cosh\left(t\sqrt{\abs{\vp_0}}\right) I +
\frac{1}{\sqrt{\abs{\vp_0}}}\sinh\left(t\sqrt{\abs{\vp_0}}\right)J & \text{if\ }
\vp_0 < 0,
\end{cases} \\[1 ex]
e^{tJ}J^{-1} & = & \begin{cases}
\frac{-\cos\left(t\sqrt{\vp_0}\right)}{\vp_0} J +
\frac{1}{\sqrt{\vp_0}}\sin\left(t\sqrt{\vp_0}\right)I & \text{if\ }
\vp_0 > 0,\\[1 ex]
\frac{-\cosh\left(t\sqrt{\abs{\vp_0}}\right)}{\vp_0} J +
\frac{1}{\sqrt{\abs{\vp_0}}}\sinh\left(t\sqrt{\abs{\vp_0}}\right)I & \text{if\ }
\vp_0 < 0.
\end{cases}
\end{eqnarray*}\eop
\end{lemma}

%

If $\vp(z_0) = 0$, then $\ker J = \lav$ so that $x_1 = x_0$ and $x_2 = 0$. The 
formulas of Lemma \ref{CPLem} reduce to 
\begin{eqnarray*}
z(t) & = & tz_0, \\
x(t) & = & tx_0.
\end{eqnarray*}
If $\vp(z_0) \neq 0$, then $\ker J = 0$ so that $x_1 = 0$ and $x_2 = x_0$. As 
mentioned above, this $x_2$ does not need to be decomposed any further since $J^2$
has only one eigenvalue. In this case, the formulas of Lemma \ref{CPLem} reduce to
\begin{eqnarray*}
z(t) & = & t\left(z_0 + \frac{1}{2}\br{J^{-1}x_0}{x_0}\right) + \frac{1}{2}
\br{e^{tJ}J^{-1}x_0}{J^{-1}x_0}, \\[0.5 ex]
x(t) & = & (e^{tJ} - I)J^{-1} x_0.
\end{eqnarray*}

Carefully applying the formulas of Lemma \ref{Jcalc}, replacing $J$ with $j(z_0)$,
and recalling from the proof of Theorem \ref{tgsg} that $\br{x_0}{j(z_0)x_0} =
\tfrac{4}{m}\norm{x_0}_\lav^2 \rc|_\laz z_0$, one obtains the following theorem.

\begin{theorem}
Suppose\label{geo} $(N,\inpe,\vp)$ is a modified $H$-type Lie group.
Let $\g :I \to N$, $\g(t) = \exp(z(t) + x(t))$, be a geodesic with $\g(0) = 1 \in N$
and $\dg(0) = z_0 + x_0 \in \n$, and suppose that both $z_0$ and $x_0$ are nonzero.
Let $\vp_0 = \vp(z_0)$.

\medskip

\nin If $\vp(z_0) = 0$, then $\g$ is given by
\begin{eqnarray}
z(t)\label{geoz0} & = & t z_0, \\
x(t)\label{geox0} & = & t x_0;
\end{eqnarray}

\nin if $\vp(z_0) > 0$, then $\g$ is given by
\begin{eqnarray}
z(t)\label{geoz+} & = & tz_0 + \left(\frac{4\sin\left(t\sqrt{\vp_0}\right) -
2t\sqrt{\vp_0}}
{m(\vp_0)^{3/2}}\right)\norm{x_0}_\lav^2\rc|_\laz z_0, \\[0.25 ex] 
x(t)\label{geox+} & = & \left(\frac{1 - \cos\left(t\sqrt{\vp_0}\right)}
{\vp_0}\right)j(z_0)x_0 + \left(\frac{\sin\left(t\sqrt{\vp_0}\right)}
{\sqrt{\vp_0}}\right)x_0;
\end{eqnarray}

\nin and if $\vp(z_0) < 0$, then $\g$ is given by
\begin{eqnarray}
z(t)\label{geoz-} & = & tz_0 + \left(\frac{4\sinh\left(t\sqrt{\abs{\vp_0}}\right) -
2t\sqrt{\abs{\vp_0}}}
{m\vp_0\sqrt{\abs{\vp_0}}}\right)\norm{x_0}_\lav^2\rc|_\laz z_0, \\[0.25 ex] 
x(t)\label{geox-} & = & \left(\frac{1 - \cosh\left(t\sqrt{\abs{\vp_0}}\right)}
{\vp_0}\right) j(z_0)x_0 
+ \left(\frac{\sinh\left(t\sqrt{\abs{\vp_0}}\right)}{\sqrt{\abs{\vp_0}}}\right)x_0.
\end{eqnarray} \eop
\end{theorem}

Recalling that it is possible for $\norm{x_0}_\lav^2\rc|_\laz z_0$ to vanish
while neither $x_0$ nor $z_0$ are zero, we obtain the following Corollary.

\begin{corollary}
Suppose\label{geoc} $(N,\inpe,\vp)$ is a modified $H$-type Lie group.
Let $\g :I \to N$, $\g(t) = \exp(z(t) + x(t))$, be a geodesic with $\g(0) = 1$
and $\dg(0) = z_0 + x_0 \in \n$, supposing that both $z_0$ and $x_0$ are nonzero.
If $\norm{x_0}_\lav^2 = 0$ or $z_0 \in \ker(\rc|_\laz)$, then the $\laz$-component
of $\g$ is simply given by $z(t) = tz_0$. \eop
\end{corollary}

Just by looking at the formulas for the geodesics in Theorem \ref{geo} it is clear
that every geodesic lives in a submanifold $M = \exp \lal{m}$, where
$\lal{m} = \span{z_0, x_0, j(z_0)x_0, \rc(z_0)}$. If $z_0$ is an eigenvector of $\rc$
corresponding to a non-zero eigenvalue and $x_0$ is non-null, then by the previous
section $M$ is a totally geodesic subgroup that is isomorphic to the Heisenberg 
group $H_3$. 

If either $x_0$ is null or $\rc(z_0) = 0$, then by Corollary \ref{tgsgc1}, 
$\br{x_0}{j(z_0)x_0} = 0 = \rc(z_0)$, $\lal{m} = \span{z_0,x_0,j(z_0)x_0}$, and $M$
is totally geodesic but not a subgroup. In fact, it is a copy of $\R^3$.

If $x_0$ is non-null and $z_0$ is not in $\ker(\rc)$ but is also not an eigenvector
of $\rc$, then $M$ is a subgroup isomorphic to $H_3(\R) \times \R$. Indeed, by the
previous section, $\br{x_0}{j(z_0)x_0} = \tfrac{4}{m}\norm{x_0}_\lav^2\rc z_0$.
Thus $z_0$ plays the role of the central extension.


\subsection{\normalsize \textit{Sectional curvatures of semi-central planes}}

In the Riemannian case, it is shown in \cite{Prip} that the generalized Heisenberg
groups of Example \ref{Rgheis} are characterized among nilpotent Lie groups by the
following property:
\begin{quote}
\textit{
Every\label{Prip1} plane $\Pi$ spanning one central and one non-central direction has
sectional curvature $K(\Pi) = c$.
}
\end{quote}
In particular, every semi-central plane $\Pi$ in an $H$-type group has sectional 
curvature $K(\Pi) = \tfrac{1}{4}$. In this section we investigate the modified $H$-%
type analogue to this result.

\medskip

Let $(N,\inpe,\vp)$ be a modified $H$-type Lie group, and $\Pi = \span{z,x}$ be a 
nondegenerate plane spanned by $z \in \laz$, $x \in \lav$. Since $\Pi$ is assumed
to be nondegenerate, we may assume that $\norm{z}^2 = \ve_z = \pm 1$ and $\norm{x}^2
= \ve_x = \pm 1$. The sectional curvature of $\Pi$ is given in \eqref{sec2} by
$$
K(\Pi) = \tfrac{1}{4} \escr{K}(z) = \tfrac{1}{4} \ve_z \vp(z).
$$
If $\vp(z) = \ve_z 4c$, $c \neq 0$, as in Example \ref{Rgheis}, then every
nondegenerate semi-central plane has curvature $K(\Pi) = c$, as in the Riemannian
case. Indeed,
\begin{equation}
K(\Pi)\label{scsec1} = \tfrac{1}{4} \ve_z \vp(z) = \tfrac{1}{4} \ve_z^2 4 c= c.
\end{equation}
In particular, nondegenerate semi-central planes in pseudo-$H$-type groups have
sectional curvature $K(\Pi) = \tfrac{1}{4}$.

The converse also holds.

\begin{theorem}
Let\label{scsec2} $(N,\inpe)$ be a 2-step nilpotent Lie group with left-invariant
metric $\inpe$ making the center nondegenerate. If the sectional curvature of every
nondegenerate semi-central plane $\Pi$ is constant, $K(\Pi) = c$, then $(N,\inpe)$ is
a modified $H$-type group with quadratic form $\vp = 4c\norm{\cdot}_\laz^2$.
\end{theorem}

\begin{proof}
Let $z \in \laz$ and $x \in \lav$ such that $\ve_z = \norm{z}_\laz^2 = \pm 1$ and
$\ve_x = \norm{x}_\lav^2 = \pm 1$, and consider the plane $\Pi = \span{z,x}$. The
sectional curvature of $\Pi$ is given in \cite{CP4} by
$$
K(z,x) = \tfrac{1}{4} \ve_z \ve_x \inp{j(z)x}{j(z)x}_\lav.
$$
Solving for $\inp{j(z)x}{j(z)x}_\lav$ yields
$$
\inp{j(z)x}{j(z)x}_\lav = 4 \ve_z \ve_x K(z,x).
$$
Recalling that $\ve_x = \inp{x}{x}_\lav$, the assumption $K(\Pi) = c$ implies that
$$
\inp{j(z)x}{j(z)x}_\lav = 4c\ve_z \inp{x}{x}_\lav
$$
for all unit vectors $z \in \laz$, $x \in \lav$. We deduce that
$j(z)^2 = -4c\ve_z\id|_\lav$ for all unit vectors in $\lav$. By linearity, since
$\laz$ is nondegenerate, we obtain that $j(z)^2 = -4c\norm{z}_\laz^2\id|_\lav$ for
all $z \in \laz$. Therefore $(N,\inpe,4c\norm{\cdot}_\laz^2)$ is a modified $H$-type
group.
\end{proof}

\begin{corollary}
The $H$-type and pseudo-$H$-type groups are
characterized, among 2-step nilpotent Lie groups, by the following property: 
Every nondegenerate semi-central plane has sectional curvature $K = \tfrac{1}{4}$.
\eop
\end{corollary}

In general, it is possible for modified $H$-type groups to have semi-central
planes that are flat. Indeed, if $\vp(z') = 0$ but $\norm{z'}^2 \neq 0$, then any 
semi-central plane in the direction of $z'$ has $K(\Pi) = 0$.

\begin{remark}
The sectional curvature of nondegenerate semi-central planes depends
only on the central direction $z$ and the quadratic form $\vp$, $K(z,x) =
\tfrac{1}{4} \ve_z\vp(z)$. Therefore, the quadratic form $\vp$ is given by
\begin{equation}
\vp(z) = 4\norm{z}_\laz^2K(z,\bar{\lav})
\end{equation}
for all non-null vectors in $\laz$, where $\bar{\lav}$ represents all non-null
vectors in $\lav$. This further illustrates the close relationship between the
quadratic form $\vp$ and the sectional curvature of $N$. 
\end{remark}

\subsection{\normalsize \textit{Nilsolitons}}

A pseudo-Riemannian metric on a 2-step nilpotent Lie group is
said to be a \emph{nilsoliton} if there exists a constant $c \in \R$ such that
the operator 
\begin{equation}
D\label{der} := \rc + c\cdot\id
\end{equation}
is a derivation of $\n$. That is, $D$ must satisfy
$$
D \br{x}{y} = \br{Dx}{y} + \br{x}{Dy}
$$
for all $x,y \in \n$.
By Theorem \ref{rop}, the Ricci operator on a modified $H$-type group
$(N,\inpe,\vp)$ is given on $\n = \laz \ds \lav$ by
$$
\rc = \begin{pmatrix}
\frac{m}{4}\id|_\vp^\dagger & 0 \\[0.5 ex]
0 & -\frac{\xi}{2}\id|_\lav
\end{pmatrix}.
$$
For $D$ to be a derivation, the following equation must be satisfied
for all $x,y \in \lav$.
\begin{equation}
\rc|_\laz(\br{x}{y})\label{nilder} = (-\xi + c)\br{x}{y}
\end{equation}
This occurs if and only if the derived algebra
$\n' = \br{\n}{\n} = \br{\lav}{\lav}$ is contained in a single eigenspace of
$\rc|_\laz$. Indeed, if $\lambda$ is the corresponding eigenvalue of $\rc|_\laz$,
then put $c = \lambda + \xi$.

We have proved the following.

\begin{theorem}
Let $(N,\inpe,\vp)$ be a modified $H$-type group. The metric $\inpe$ is a
nilsoliton if and only if the derived algebra $\br{\lav}{\lav}$ is contained in
a single eigenspace of the Ricci operator $\rc|_\laz$. \eop
\end{theorem}

As a corollary we recover a well known result about (pseudo-) $H$-type groups
\cite{L01,OP}.

\begin{corollary}
Every $H$-type and pseudo-$H$-type metric is a nilsoliton.
\end{corollary}

\begin{proof}
The Ricci operator for an $H$-type or pseudo-$H$-type group is given by
$$
\rc = \begin{pmatrix}
\frac{m}{4}\id|_\laz & 0 \\
0 & -\frac{p}{2}\id|_\lav
\end{pmatrix},
$$
so that every $z \in \laz$ is and eigenvector of $\rc|_\laz$.
\end{proof}

Given a nilsoliton of modified $H$-type, we can construct many.

\begin{corollary}
If a modified $H$-type group admits a nilsoliton, then every trivial central extension
also admits a nilsoliton.
\end{corollary}

\begin{proof}
Let $(N,\inpe,\vp)$ be a modified $H$-type group such that $\inpe$ is a nilsoliton.
Let $\lal{m}$ be a real vector space endowed with an inner product $\inpe_\lal{m}$
of any signature, and consider the trivial central extension of $\n$ given by
$\bar{\n} = \n \ds \lal{m}$. The simply connected Lie group $\bar{N} = \exp \bar{\n}$
endowed with the left-invariant metric $\inpe + \inpe_\lal{m}$ is a modified $H$-type 
algebra with $\vp(\lal{m}) = 0$. The Ricci operator on the center remains 
unchanged on $\laz$, so $\br{\lav}{\lav}$ remains contained in a single eigenspace of
$\rc$.
\end{proof}

\section*{Acknowledgements}  

This work was started while the author was a postdoctoral research fellow with 
CONICET at Universidad Nacional de Rosario, in Rosario, Argentina. The author would like 
to thank Gabriela Ovando and Phil Parker for their guidance and support.

\vfill

\nin Justin M. Ryan \\
\nin 3 August 2021 \\
\nin {\sf justin@geometerjustin.com}
\end{document}